\documentclass[envcountsect,oribibl,orivec]{llncs}
\usepackage{amssymb}
\usepackage{amsmath}
\usepackage[right=4cm,left=4cm]{geometry}

\usepackage{bussproofs}

\EnableBpAbbreviations
\usepackage{graphics}
 \usepackage{graphicx}

  \usepackage{qtree}
\usepackage{tikz}

\makeatletter
\newcommand{\myLines}[1]{
\begin{picture}(4,1)
\put(0,0){\line(2,1){2}}
\put(2,0){\vector(0,1){0.7}}
\put(4,0){\line(-2,1){2}}
\end{picture}}
\let\qdrawReal=\qdraw@branches
\newcommand\brOverride{\let\qdraw@branches=\myLines}
\newcommand\brRestore{\let\qdraw@branches=\qdrawReal}
\makeatother
\newcommand{\JL}{{\sf JL}}
\newcommand{\LP}{{\sf LP}}
\newcommand{\M}{{\mathcal M}}
\newcommand{\E}{{\mathcal E}}

\newcommand{\V}{{\mathcal V}}
\newcommand{\CS}{\mathcal{CS}}

 \def\r{\rightarrow}
\def\di{\displaystyle}

\begin{document}
\title{Tableau Proof Systems for Justification Logics}

\author{Meghdad Ghari}
\institute{School of Mathematics,
Institute for Research in Fundamental Sciences (IPM), \\ P.O.Box: 19395-5746, Tehran, Iran \\ \email{ghari@ipm.ir}
}

\maketitle
\begin{abstract}
In this paper we present tableau proof systems for various justification logics. We show that the tableau systems are sound and complete with respect to Mkrtychev models. In order to prove the completeness of the tableaux, we give a syntactic proof of cut elimination. We also show the subformula property for our tableaux. \\

\end{abstract}

\section{Introduction}

Justification logics are modal-like logics that provide a framework for reasoning about epistemic justifications (see \cite{A2008,ArtemovFitting,Fitting2008} for a survey). The language of justification logics extends the language of propositional logic by justification terms and expressions of the form $t:A$, with the intended meaning ``$t$ is a justification for $A$''. Justification terms are constructed from variables and constants by means of various operations. The first logic in the family of justification logics, \emph{the Logic of Proofs} \LP, was introduced by  Artemov in \cite{A1995,A2001}. The logic of proofs is a counterpart of modal logic {\sf S4}. Other logics of this kind have been introduced so far (cf. \cite{KuznetsStuder2012}).  In this paper we deal only with those justification logics which are counterparts of normal modal logics between {\sf K} and {\sf S5}.

Various tableau proof systems have been developed for the logic of proofs (see
\cite{Finger2010,Fitting2005,Renne2004,Renne2006}). However, it seems the only
analytic tableau proof system is Finger's KE tableaux for the logic of proofs \cite{Finger2010}. Finger's tableau system has KE tableau rules (cf. \cite{D’AgostinoMondadori1994}) in its propositional part. KE tableaux have linear tableau rules for propositional connectives, and the cut rule $(PB)$. 

 Most of the justification logics still lack tableau proof
systems. The aim of this paper is to present tableau proof systems for various justification logics. For each justification logic we present two tableau proof systems.  All tableau proof systems are sound and complete with respect to Mkrtychev models of justification logics.

In the first formulation (see Section \ref{sec:Tableaux 1}), the rules of the tableau system for {\sf J} is similar to the ({\sf J}-part) tableau rules given by Renne in \cite{Renne2006} for \LP. Renne's tableaux corresponds to the Artemov's sequent calculus for $\LP$ in \cite{A2001}. The subformula property fails for both the tableaux and the sequent calculus of \LP, and also fails for the tableaux of justification logics introduced in this section.

In the second formulation (see Section \ref{sec:JL^T tableaux}), we present a tableau system for {\sf JL},  which is similar to its KE tableau system but with ordinary propositional rules.
 Our propositional tableau rules are the ordinary ones given by Smullyan \cite{Smullyan1968}, and justification tableau rules are  similar to those introduced by Finger \cite{Finger2010}. In order to prove the completeness of these tableaux, we give a syntactic proof of cut elimination. Following Finger \cite{Finger2010}, by restricting the applications of $(PB)$ to analytic ones, we obtain analytic tableaux for justification logics. We give a definition of subformulas in the context of justification logics, and prove that our tableau systems enjoy the subformula property.   
 
\section{Justification logics}\label{sec:Justification logics}
The language of justification logics is an
extension of the language of propositional logic by the formulas
of the form $t:F$, where $F$ is a formula and $t$ is a
justification term. \textit{Justification terms} (or
\textit{terms} for short) are built up from (justification)
variables $x, y, z, \ldots$ and (justification) constants $a,b,c,\ldots$  using several operations depending on the logic: (binary) application `$\cdot$', (binary) sum `$+$', (unary) verifier `$!$', (unary) negative verifier `$?$', and (unary) weak negative verifier `$\bar{?}$'. Subterms of a term are defined in the usual way: $s$ is a subterm of $s,
s+t, t+s, s\cdot t$, $!s$, $\bar{?}s$, and  $?s$.

Justification formulas are constructed from a countable set of propositional variables, denoted $\mathcal{P}$, by the following grammar:
\[ A::= p~|~\bot~|~\neg A~|~A\rightarrow A~|~t:A,\]

where $p\in\mathcal{P}$ and $t$ is a justification term. Other Boolean connectives are defined as usual.

We now begin with describing the axiom schemes and rules of the basic
justification logic {\sf J}, and continue with other justification
logics. The basic justification logic {\sf J} is the weakest
justification logic we shall be discussing. Other
justification logics are obtained by adding certain axiom schemes
to {\sf J}.
\begin{definition}\label{def: justification logics}
Axioms schemes of {\sf J} are:
\begin{description}
\item[Taut.] All propositional tautologies,
\item[Sum.] $s:A\rightarrow (s+t):A~,~s:A\rightarrow (t+s):A$,
\item[jK.] $s:(A\rightarrow B)\rightarrow(t:A\rightarrow (s\cdot t):B)$.
\end{description}
Other justification logics are obtained by adding the following axiom schemes to {\sf J} in various combinations:
\begin{description}
\item[jT.] $t:A\rightarrow A$.
\item[jD.] $t:\perp \rightarrow \perp$.
\item[j4.] $t:A\rightarrow !t:t:A$,
\item[jB.] $\neg A\rightarrow\bar{?} t:\neg t: A$.
\item[j5.] $\neg t:A\rightarrow ?t:\neg t:A$.
\end{description}
All justification logics have the inference rule Modus Ponens, and the \textit{Iterated Axiom Necessitation} rule:
\begin{description}
\item[IAN.]
$\vdash c_{i_n}:c_{i_{n-1}}:\ldots:c_{i_1}:A$, where $A$ is an axiom instance of the logic, $c_{i_j}$'s
are arbitrary justification constants and $n\geq 1$.
\end{description}
\end{definition}
In what follows, {\sf JL} denotes any of the justification logics defined in Definition \ref{def: justification logics}, unless stated otherwise. The language of each justification logic $\JL$ includes those operations on terms that are present in its axioms. $Tm_\JL$ and $Fm_\JL$ denote the set of all terms and the set of all formulas of $\JL$ respectively. Moreover, the name of each justification logic is indicated by the list of its axioms. For example, ${\sf JT4}$ is the extension of ${\sf J}$ by axioms jT and j4, in the language containing term operations $\cdot$, $+$, and $!$. {\sf JT4} is usually called the logic of proofs $\LP$.

\begin{definition}
A \textit{constant specification} $\CS$
for \JL~is a set of formulas of the form
$c_{i_n}:c_{i_{n-1}}:\ldots:c_{i_1}:A$, where $n\geq 1$, $c_{i_j}$'s are
justification constants and $A$ is an axiom instance of \JL, such that it is downward closed: if $c_{i_n}:c_{i_{n-1}}:\ldots:c_{i_1}:A\in\CS$, then $c_{i_{n-1}}:\ldots:c_{i_1}:A\in\CS$.
\end{definition}

The typical form of a formula in a constant specification for \JL~is $c:F$, where $c$ is a justification constant, and $F$ is either an axiom instance of \JL~or of the form $c_{i_m}:c_{i_{m-1}}:\ldots:c_{i_1}:A$, where $m\geq 1$, $c_{i_j}$'s are justification constants and $A$ is an axiom instance of \JL.

Let ${\sf JL}_\CS$ be the fragment of ${\sf JL}$ where the Iterated Axiom Necessitation rule only produces formulas from the given $\CS$.

In the remaining of this section, we recall the definitions of M-models for justification logics (see \cite{Mkrtychev1997, KuznetsStuder2012}).

\begin{definition}\label{Kripke-Fitting models J}
 An M-model $\M=(\E, \V)$ for justification
logic ${\sf J}_\CS$ consists of a valuation $\V:\mathcal{P} \rightarrow \{0,1\}$ and an admissible evidence function $\E:Tm_\JL \rightarrow 2^{Fm_\JL}$ meeting the following conditions:
\begin{description}
 \item[$\E 1.$]  $A\r B\in\E(s)$ and $A\in\E(t)$ implies $B\in\E(s\cdot t)$.
  \item[$\E 2.$]  $\E(s)\cup \E(t)\subseteq\E(s+t)$.
  \item[$\E 3.$]  $c:F\in\CS$ implies $F\in\E(c)$.
  \end{description}
  \end{definition}

\begin{definition}\label{def:forcing relation}
For an M-model $\M=(\E, \V)$ the forcing relation $\Vdash$ is defined as follows:
\begin{enumerate}
\item $\M\not\Vdash \bot$,
\item $\M\Vdash p$ if{f} $\V(p)=1$, for $p\in\mathcal{P}$,
 \item $\M\Vdash \neg A$ if{f} $\M\not\Vdash A$,
  \item $\M\Vdash A\r B$ if{f} $\M\not\Vdash A$ or $\M\Vdash B$,
 \item If $\JL$ does not contain axiom jT: $\M\Vdash t:A$ if{f} $A\in\E(t)$.

 If $\JL$ contains axiom jT: $\M\Vdash t:A$ if{f} $A\in\E(t)$ and $\M \Vdash A$.
\end{enumerate}
If $\M\Vdash F$ then it is said that $F$ is true in $\M$ or $\M$ satisfies $F$.
\end{definition}

In order to define M-models for other justification logics of Definition \ref{def: justification logics} certain
additional conditions should be imposed on the M-model.

\begin{definition}\label{Kripke-Fitting models JL}
 An M-model $\M=(\E, \V)$ for justification logic ${\sf JL}_\CS$ is an M-model for ${\sf J}_\CS$ such that:
\begin{itemize}
\item if $\JL$ contains axiom jD, then for all $t\in Tm_\JL$:
\begin{description}
   \item[$\E 4.$]  $\bot\not\in\E(t)$.\vspace{0.05cm}
     \end{description}
\item if $\JL$ contains axiom j4, then for all $t\in Tm_\JL$ and $A\in Fm_\JL$:
\begin{description}
   \item[$\E 5.$]  $A\in\E(t)$ implies $t:A\in\E(!t)$.\vspace{0.05cm}
     \end{description}
\item if $\JL$ contains axiom jB, then for all $t\in Tm_\JL$ and $A\in Fm_\JL$:
\begin{description}
     \item[$\E 6.$]  $\M\not\Vdash A$ implies $\neg t:A\in\E(\bar{?}t)$.\vspace{0.05cm}
    \end{description}
\item if $\JL$ contains axiom j5, then for all $t\in Tm_\JL$ and $A\in Fm_\JL$:
\begin{description}
       \item [$\E 7.$] $A\not\in\E(t)$ implies $\neg t:A\in\E(?t)$.
         \end{description}
\end{itemize}
\end{definition}

By a $\JL_\CS$-model we mean an M-model for justification logic $\JL_\CS$. A \JL-formula $F$ is $\JL_\CS$-valid if it is true in every $\JL_\CS$-model. For a set $S$ of formulas, $\M\Vdash S$ provided that $\M\Vdash F$ for all formulas $F$ in $S$. Note that given a constant specification $\CS$ for \JL, and a model $\M$ of $\JL_\CS$ we have $\M\Vdash \CS$ (in this case it is said that $\M$ respects $\CS$).

The proof of soundness and completeness theorems for all justification logics of Definition \ref{def: justification logics} are given in \cite{KuznetsPhD2008,KuznetsStuder2012}.

\begin{theorem}\label{thm:Sound Compl JL}
Let \JL~be one of the justification logics of Definition \ref{def: justification logics}, and $\CS$ be a constant specification for \JL. Then a \JL-formula $F$ is provable in  $\JL_\CS$ if{f} $F$ is $\JL_\CS$-valid.
\end{theorem}

\section{Tableaux}

In this section we present two different tableau proof systems for each justification logic of Definition \ref{def: justification logics}.  The rules of our first tableau system for {\sf J} in Section \ref{sec:Tableaux 1} is similar to that given in \cite{Renne2006}. In Section \ref{sec:JL^T tableaux} we present a tableau system for {\sf J} which is similar to the KE tableau system of $\LP$ in \cite{Finger2010}, but with ordinary propositional rules instead of linear propositional KE rules. 

\subsection{$\JL$-Tableaux}\label{sec:Tableaux 1}
Tableau proof systems for the logic of proofs are given in \cite{Fitting2005, Renne2004,Renne2006}. In this section we present similar tableaux for all justification logics.

A ${\sf J}_\CS$-tableau for a formula is a binary tree with the negation of  that formula at the root constructed by applying ${\sf J}_\CS$-tableau rules from Table \ref{table:tableau rules J}. For extensions of {\sf J}, tableau rules corresponding to axioms from Table \ref{table:tableau rules JL} should be added to ${\sf J}_\CS$-tableau rules. For example, the tableau proof system of the logic of proofs $\LP$ is obtained by adding the rules $(T:)$ and $(F!)$ to the tableau rules of {\sf J}. For a justification logic \JL, a tableau branch of a $\JL_\CS$-tableau closes if one of the following holds:
\begin{enumerate}
\item Both $A$ and $\neg A$ occurs in the branch, for some formula $A$.
\item $\bot$ occurs  in the branch.
\item $\neg c:F$ occurs in the branch, for some $c:F\in\CS$.
\end{enumerate}

A tableau closes if all branches of the tableau close. A $\JL_\CS$-tableau proof for formula $F$ is a closed tableau beginning with $\neg F$ (the root of the tableau) using only tableau rules of $\JL_\CS$. A $\JL_\CS$-tableau for a finite set $S$ of $\JL$-formulas begins with a single branch whose nodes consist of the formulas of $S$ as roots.

\begin{example}
We give a ${\sf J}_\CS$-tableau proof of $x:A \r c \cdot x:(B\r A)$, where $\CS$ contains $c:(A\r (B\r A))$.
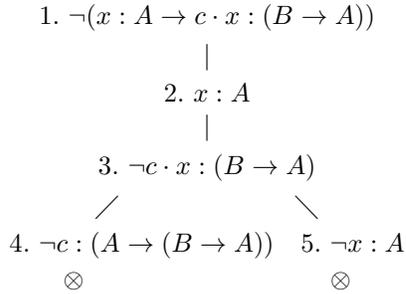
\begin{prooftree}
\alwaysNoLine
\AXC{$1.~\neg(x:A \r c \cdot x:(B\r A))$}
\UIC{$|$}
\UIC{$2.~x:A$}
\UIC{$|$}
\UIC{$3.~\neg c \cdot x:(B\r A)$}
\UIC{$\diagup~~~~~~~~~~~~~~~~~~~~\diagdown$}
\UIC{$4.~\neg c:(A\r (B\r A))~~~5.~\neg x:A$}
\UIC{$\otimes~~~~~~~~~~~~~~~~~~~~~~~~~~~~\otimes$}
\end{prooftree}
Formulas 2 and 3 are from 1 by rule $(F\r)$, and 4 and 5 are from 3  by rule $(F\cdot)$. Closed branches are indicated by $\otimes$.
\end{example}

\begin{table}[ht]
\centering\renewcommand{\arraystretch}{1.2}
\begin{tabular}{|lc|}
\hline
~Propositional rules: &\\
\multicolumn{2}{|c|}{
\AXC{$\neg\neg A$}\RightLabel{$(F\neg)$}
\UIC{$A$}\noLine
\UIC{}
\DP
}\\

\AXC{$\neg(A\rightarrow B)$}\RightLabel{$(F\r)$}
\UIC{$A$}\noLine
\UIC{$\neg B$}\noLine
\UIC{}
\DisplayProof
&
\AXC{$A\rightarrow B$}\RightLabel{$(T\r)$}
\UIC{$\neg A | B$}
\DP
\\
~Justification rules:&\\
\AXC{$\neg t+s:A$}\RightLabel{$(F+)$}
\UIC{$\neg t:A$}\noLine
\UIC{$\neg s:A$}\noLine
\UIC{}
\DisplayProof

&
\AXC{$\neg s\cdot t:B$}
\RightLabel{$(F\cdot)$}
\UIC{$\neg s:(A\rightarrow B) | \neg t:A$}\noLine
\UIC{}
\DisplayProof
\\\hline
\end{tabular}\vspace{0.3cm}
 \caption{Tableau rules for basic justification logic ${\sf J}$.}\label{table:tableau rules J}
\end{table}

\begin{table}[ht]
\centering\renewcommand{\arraystretch}{2}
\begin{tabular}{|l|c|}
\hline
 ~Justification axiom & Tableau rule~ \\\hline
 ~{\bf jT}. $t:A\r A$ &
 \AXC{$t:A$} \RightLabel{$(T:)$}
 \UIC{$A$}
 \DP \\\hline
 ~{\bf jD}. $t:\bot\r\bot$  &
 \AXC{} \RightLabel{$(F:_\bot)$}
 \UIC{$\neg t:\bot$}
 \DP
  \\\hline
 ~{\bf j4}. $t:A\r !t:t:A$  &
 \AXC{$\neg !t:t:A$} \RightLabel{$(F!)$}
 \UIC{$\neg t:A$}
 \DP
 \\\hline
 ~{\bf jB}. $\neg A\r \bar{?}t:\neg t:A$ &
 \AXC{$\neg \bar{?}t:\neg t:A$} \RightLabel{$(F\bar{?})$}
 \UIC{$A$}
 \DP
  \\\hline
 ~{\bf j5}. $\neg t:A\r ?t:\neg t:A$ &
 \AXC{$\neg ?t:\neg t:A$} \RightLabel{$(F?)$}
 \UIC{$t:A$}
 \DP
  \\
  \hline
\end{tabular}\vspace{0.3cm}
\caption{Justification axioms with corresponding
tableau rules.}\label{table:tableau rules JL}
\end{table}

Let us show the soundness and completeness of tableau systems with respect to M-models. Our starting point is the following lemma, whose proof is straightforward and is omitted here.

\begin{lemma}\label{lem: soundness lemma}
Let $\pi$ be any branch of a $\JL_\CS$-tableau and $\M$ be a $\JL_\CS$-model that satisfies all the formulas occur in $\pi$. If a $\JL_\CS$-tableau rule is applied to $\pi$, then it produces at least one extension $\pi'$ such that $\M$ satisfies all the formulas occur in $\pi'$.
\end{lemma}

\begin{theorem}[Soundness]
If $A$ has a $\JL_\CS$-tableau proof, then it is $\JL_\CS$-valid.
\end{theorem}
\begin{proof}
If $A$ is not $\JL_\CS$-valid, then there is a $\JL_\CS$-model $\M$ such that $\M\Vdash \neg A$. Thus by Lemma \ref{lem: soundness lemma}, there is no closed $\JL_\CS$-tableau beginning with $\neg A$. Therefore, $A$ does not have a $\JL_\CS$-tableau proof.\qed
\end{proof}

Next we shall prove the completeness theorem, by making use of maximal consistent sets.

\begin{definition}
Suppose $\Gamma$ is a set of $\JL$-formulas. $\Gamma$ is (tableau) $\JL_\CS$-consistent if there is no closed tableau beginning with any finite subset of $\Gamma$. $\Gamma$ is maximal if it has no proper tableau consistent extension.
\end{definition}

It is known  that every $\JL_\CS$-consistent set has a maximally $\JL_\CS$-consistent extension (Lindenbaum Lemma).

It is easy to show that maximally $\JL_\CS$-consistent sets are  closed under $\JL_\CS$-tableau rules. For a non-branching rule like
\begin{prooftree}
\AXC{$\alpha$}
\UIC{$\alpha_1$}\noLine
\UIC{$\alpha_2$}
\end{prooftree}
this means that if $\alpha$ is in a maximally $\JL_\CS$-consistent set $\Gamma$, then both $\alpha_1\in\Gamma$ and $\alpha_2\in\Gamma$. For a branching rule like
\[\di{\frac{\beta}{\beta_1 | \beta_2}}\]
this means that if $\beta$ is in a maximally $\JL_\CS$-consistent set $\Gamma$, then $\beta_1\in\Gamma$ or $\beta_2\in\Gamma$. For the rule $(F\cdot)$ this means that if $\neg s \cdot t :  B\in\Gamma$, then for every formula $A$  either $\neg s: (A\r B) \in\Gamma$ or $\neg t: A \in\Gamma$.

\begin{lemma}\label{lem:downward saturated}
Suppose $\Gamma$ is a maximally $\JL_\CS$-consistent set. Then $\Gamma$ is closed under $\JL_\CS$-tableau rules.
\end{lemma}
\begin{proof}
The proof for propositional rules $(F\neg)$, $(F\r)$, and $(T\r)$ are standard. For justification rules,  we detail the proof only for the rules $(F\cdot)$ and $(F:_\bot)$. The proof for the other tableau justification rules is similar.

For $(F\cdot)$, suppose $\Gamma$ is a maximally $\JL_\CS$-consistent set and  $\neg s \cdot t :B\in\Gamma$. Suppose towards a contradiction that for some formula $A$ we have $\neg s:(A\r B)\not\in\Gamma$ and $\neg t:A\not\in\Gamma$. Since $\Gamma$ is maximal, we have $\Gamma \cup \{ \neg s: (A\r B) \}$ and $\Gamma \cup \{ \neg t: A\}$ are not tableau $\JL_\CS$-consistent. Thus there are closed $\JL_\CS$-tableaux for finite subsets, say $\Gamma_1 \cup \{ \neg s: (A\r B) \}$ and $\Gamma_2 \cup \{ \neg t:  A\}$. But $\Gamma_1 \cup \Gamma_2 \cup \{\neg s \cdot t :  B \}$ is a finite subset of $\Gamma$ and, using rule $(F\cdot)$, there is a closed $\JL_\CS$-tableau for it, contra the tableau $\JL_\CS$-consistency of $\Gamma$.

For $(F:_\bot)$, suppose towards a contradiction that $\neg t:\bot \not\in \Gamma$, for some term $t$. Then, $\Gamma \cup \{ \neg t:\bot \}$ is not tableau $\JL_\CS$-consistent. Thus, there is a closed $\JL_\CS$-tableau for a finite subset, say $\Gamma_0 \cup \{ \neg t:\bot \}$. Using rule $(F:_\bot)$, there is a closed $\JL_\CS$-tableau for  $\Gamma_0$,  contra the tableau $\JL_\CS$-consistency of $\Gamma$. Therefore, $\neg t:\bot \in \Gamma$, for any term $t$.\qed
\end{proof}

\begin{definition}\label{def:canonical model tableau}
Given a maximally $\JL_\CS$-consistent set $\Gamma$, the canonical model $\M=(\E,\V)$ with respect to $\Gamma$ is defined as follows:
\begin{itemize}
\item $\E(t) = \{A~|~\neg t:A \not\in \Gamma\}$.
\item $\V(p)=1$ if{f} $p\in\Gamma$, where $p\in\mathcal{P}$.
\end{itemize}
\end{definition}

\begin{lemma}[Truth Lemma]
Suppose $\Gamma$ is a maximally $\JL_\CS$-consistent set and $\M=(\E,\V)$ is the canonical model with respect to $\Gamma$. Then for every $\JL$-formula $F$:

\begin{enumerate}
\item $F\in\Gamma$ implies $\M\Vdash F$.

\item $\neg F\in\Gamma$ implies $\M \not\Vdash F$.
\end{enumerate}
\end{lemma}
\begin{proof}
By induction on the complexity of $F$. The base case and the propositional inductive cases are standard. The proof for the case that $F=t:A$ is as follows.

Suppose that $t:A\in\Gamma$. Since $\Gamma$ is tableau $\JL_\CS$-consistent, $\neg t:A \not\in \Gamma$. Thus $A \in \E(t)$. If $\JL$ does not contain axiom jT, then $\M\Vdash t:A$ as desired.  If $\JL$ contains axiom jT, then since $\Gamma$ is closed under $(T:)$, $A \in \Gamma$. Thus, by the induction hypothesis, $\M\Vdash A$. Hence $\M\Vdash t:A$.

Suppose that $\neg t:A\in\Gamma$. Thus $A \not\in \E(t)$, and hence   $\M \not\Vdash t:A$. \qed
\end{proof}

\begin{lemma}
Given a maximally $\JL_\CS$-consistent set $\Gamma$, the canonical model $\M=(\E,\V)$ with respect to $\Gamma$ is a $\JL_\CS$-model.
\end{lemma}
\begin{proof}
Suppose $\Gamma$ is a maximally $\JL_\CS$-consistent set and $\M=(\E,\V)$ is the canonical model with respect to $\Gamma$. We  shall show that the admissible evidence function $\E$ satisfies the corresponding conditions stated in the definition of $\JL_\CS$-models.

For $\E 1$, suppose that $A\in\E(t)$ and $A\r B\in\E(s)$. We have to show that $B\in\E(s\cdot t)$. By the definition of $\E$,   $\neg t:A\not\in\Gamma$ and $\neg s:(A\r B)\not \in\Gamma$. By Lemma \ref{lem:downward saturated}, $\Gamma$ is closed under rule $(F\cdot)$, and hence $\neg s\cdot t:B\not\in\Gamma$. Hence, by the definition of $\E$, $B\in\E(s\cdot t)$.

For  $\E 2$, suppose that $A\in\E(s)\cup \E(t)$. We have to show that $A\in\E(s+t)$. If $A\in\E(s)$, then $\neg s:A \not\in\Gamma$. By Lemma \ref{lem:downward saturated}, $\Gamma$ is closed under rule $(F+)$, and hence $\neg s+t:A\not\in\Gamma$. Therefore,  $A\in\E(s+t)$. The case that $A\in\E(t)$ is similar.

For  $\E 3$, suppose that $c:F\in\CS$. We have to show that $F\in\E(c)$. Since $\Gamma$ is $\JL_\CS$-consistent, $\neg c:F\not\in\Gamma$. Thus  $F\in\E(c)$.

For  $\E 4$, where $\JL$ contains axiom jD, by Lemma \ref{lem:downward saturated} we have $\neg t:\bot \in\Gamma$ for any term $t\in Tm_\JL$.   Thus  $\bot\not\in\E(t)$.

For  $\E 5$, where $\JL$ contains axiom j4, suppose that $A\in\E(t)$. We have to show that $t:A\in\E(!t)$. By the definition of $\E$, $\neg t:A \not\in\Gamma$. By Lemma \ref{lem:downward saturated}, $\Gamma$ is closed under rule $(F!)$, and hence $\neg !t:t:A\not\in\Gamma$. Therefore,  $t:A\in\E(!t)$.

For  $\E 6$, where $\JL$ contains axiom jB, suppose that $\M\not\Vdash A$. We have to show that $\neg t:A\in\E(\bar{?} t)$. By the Truth Lemma, $A\not\in\Gamma$. By Lemma \ref{lem:downward saturated}, $\Gamma$ is closed under rule $(F\bar{?})$, and hence $\neg \bar{?} t:\neg t:A\not\in\Gamma$. Therefore, $\neg t:A\in\E(\bar{?} t)$.

For  $\E 7$, where $\JL$ contains axiom j5, suppose that $A\not\in\E(t)$. We have to show that $\neg t:A\in\E(? t)$. By the definition of $\E$,  $\neg t: A \in\Gamma$. By Lemma \ref{lem:downward saturated}, $\Gamma$ is closed under rule $(F?)$, and hence $\neg ?t:\neg t:A\not\in\Gamma$. Therefore,  $\neg t:A\in\E(?t)$.\qed
\end{proof}

\begin{theorem}[Completeness]\label{thm:completeness tableaux}
If $A$ is $\JL_\CS$-valid, then it has a $\JL_\CS$-tableau proof.
\end{theorem}
\begin{proof}
If $A$ does not have a $\JL_\CS$-tableau proof, then $\{\neg A\}$ is a $\JL_\CS$-consistent set and can be extended to a maximal $\JL_\CS$-consistent set $\Gamma$. Since $\neg A\in\Gamma$, by the Truth Lemma, $\M\not\Vdash A$, where $\M$ is the canonical model of $\JL_\CS$ with respect to $\Gamma$. Therefore $A$ is not $\JL_\CS$-valid.\qed
\end{proof}

Clearly in any $\JL_\CS$-tableau system the rule $(F\cdot)$ 
\[
\AXC{$\neg s\cdot t:B$}
\RightLabel{$(F\cdot)$}
\UIC{$\neg s:(A\rightarrow B) | \neg t:A$}\noLine
\UIC{}
\DisplayProof
\]
is not analytic, because the formula $A$ in the conclusion of the rule could be a new formula from the outside of the proof. The rule $(F:_\bot)$ is not analytic too. In the following section we replace these rules with analytic rules.
\subsection{$\JL^\mathcal{T}$-tableaux}\label{sec:JL^T tableaux}

In this section we present analytic tableaux for justification logics. The rule $(F\cdot)$ is replaced with the analytic non-branching rule  $(T\cdot)$ (see Table \ref{table:J^T}) and $(F:_\bot)$ is replaced with an analytic rule. The rule $(T\cdot)$ was introduced by Finger in \cite{Finger2010} in a tableau proof system for the logic of proofs based on KE tableaux (cf. \cite{D’Agostino1992,D’Agostino1999,D’AgostinoMondadori1994}).\footnote{It is worth noting that Finger's completeness proof of KE tableaux for the logic of proofs in \cite{Finger2010} contains a mistake. In fact, he wrongly claimed that every $\LP$-tableau proof (see Section \ref{sec:Tableaux 1}) can be simulated by KE tableaux of $\LP$. Then he used the completeness of $\LP$-tableaux to show that KE tableau system of $\LP$ is complete.} The tableau proof system of this section is similar to KE tableaux, with the difference that its propositional logic rules is the same as Smullyan's rules \cite{Smullyan1968}. A restricted form of the cut rule, called the principle of bivalence in \cite{D’Agostino1992,D’Agostino1999,D’AgostinoMondadori1994} and denoted by $(PB)$, is also added to the rules. In order to make the rules $(T\cdot)$ and $(PB)$ analytic we put some restrictions on the application of these rules. Let us first extend the definition of subformulas of a  formula to include constant specifications.

\begin{definition}\label{def:subformula}
Let $\CS$ be a constant specification for $\JL$, and let $A$ and $B$ be $\JL$-formulas. $A$ is a $\JL_\CS$-subformula of $B$ if one of the following clauses holds:

\begin{enumerate}
\item $A=B$,

\item $B=\neg F$, and $A$ is a $\JL_\CS$-subformula of $F$,

\item $B=F\r G$, and $A$ is a $\JL_\CS$-subformula of $F$ or $G$,

\item $B=t:F$, and $A$ is a $\JL_\CS$-subformula of $F$,

\item $A=t:F$, where $t$ is a subterm of a term in $B$ and $F$ is a $\JL_\CS$-subformula of $B$,

\item $A$ is a $\JL_\CS$-subformula of  $c_{i_n}:c_{i_{n-1}}:\ldots:c_{i_1}:F\in\CS$, where $F$ is an axiom instance of $\JL$.

\item The relation of ``$\JL_\CS$-subformula of", defined in clauses 1-6, is extended by transitivity.
 \end{enumerate}

 $A$ is a weak $\JL_\CS$-subformula of $B$ if  $A$ is either a $\JL_\CS$-subformula of $B$ or the negation of a $\JL_\CS$-subformula of $B$.
\end{definition}

Tableau rules for basic justification logic {\sf J} are given in Table \ref{table:J^T}. We denote this tableau system by ${\sf J}^\mathcal{T}$. For extensions of {\sf J}, tableau rules corresponding to axioms from Table \ref{table:tableau rules JL} should be added to the rules of ${\sf J}^\mathcal{T}$, except that in those justification logics that contain axiom jD the rule $(F:_\bot)$ is replaced by the following rule: 

\[ \di{\frac{t:\bot}{\bot}} (T:_\bot)\]

 The closure conditions are the same as $\JL$-tableaux. For a justification logic $\JL$, the resulting tableau system is denoted by $\JL^\mathcal{T}$.

 Note that in $\JL^\mathcal{T}$-tableaux the rules $(T\cdot)$ and $(PB)$ have restrictions on their applications (see Table \ref{table:J^T}). The formula $A$ in the conclusion of $(PB)$ is called the $PB$-formula. Furthermore, the rule $(T\cdot)$ is a binary rule (it takes two formulas as input), and it should be read as follows: if a branch contains $s:(A\rightarrow B)$ and $t:A$, then we can extend that branch by adding $s\cdot t:B$, provided that the formulas $s:(A\rightarrow B)$, $t:A$, and $s\cdot t:B$ are all $\JL_\CS$-subformulas of the root of the tableau. In addition, there is no ordering intended on the input $s:(A\rightarrow B)$, $t:A$.

\begin{table}
\centering\renewcommand{\arraystretch}{1.5}
\begin{tabular}{|p{1.5in}p{1.5in}p{1in}|}
\hline
\multicolumn{3}{|l|}{~Propositional rules:}\\
\AXC{$\neg\neg A$}\RightLabel{$(F\neg)$}
\UIC{$A$}\noLine
\UIC{}
\DP
&
 \AXC{$\neg(A\rightarrow B)$}\RightLabel{$(F\r)$}
\UIC{$A$}\noLine
\UIC{$\neg B$}\noLine
\UIC{}
\DisplayProof
&
\AXC{$A\rightarrow B$}\RightLabel{$(T\r)$}
\UIC{$\neg A | B$}
\DP
\\
\multicolumn{3}{|l|}{~Justification rules:}\\
\AXC{$\neg t+s:A$}\RightLabel{$(F+_L)$}
\UIC{$\neg t:A$}\noLine
\UIC{}
\DisplayProof

&
\AXC{$\neg t+s:A$}\RightLabel{$(F+_R)$}
\UIC{$\neg s:A$}\noLine
\UIC{}
\DisplayProof
&
\AXC{}\noLine
\UIC{$s:(A\rightarrow B)$}\noLine
\UIC{$t:A$}\RightLabel{$(T\cdot)$}
\UIC{$s\cdot t:B$}\noLine
\UIC{}
\DisplayProof
\\
\multicolumn{3}{|l|}{~Principle of Bivalence:}\\
\multicolumn{3}{|c|}{
\AXC{}\RightLabel{$(PB)$}
\UIC{$A~|~\neg A$}\noLine
\UIC{}
\DisplayProof}\\\hline
\multicolumn{3}{|l|}{ \parbox[b]{33em}{ In $(T\cdot)$ the formulas $s:(A\rightarrow B)$, $t:A$, and $s\cdot t:B$ are all  $\JL_\CS$-subformulas of the root of the tableau.}
}\\ \hline 
\multicolumn{3}{|l|}{In $(PB)$ the $PB$-formula $A$ is a $\JL_\CS$-subformula of the root of the tableau.}\\\hline
\end{tabular}\vspace{0.3cm}
 \caption{Tableau rules of ${\sf J}^\mathcal{T}$ for basic justification logic {\sf J}.}\label{table:J^T}
\end{table}

From Definition \ref{def:subformula} it is obvious that the following is an instance of $(PB)$:

\begin{prooftree}
\AXC{}\RightLabel{$(PB)$}
\UIC{$c_{i_n}:c_{i_{n-1}}:\ldots:c_{i_1}:A~|~\neg c_{i_n}:c_{i_{n-1}}:\ldots:c_{i_1}:A$}
\end{prooftree}

where $c_{i_n}:c_{i_{n-1}}:\ldots:c_{i_1}:A\in\CS$. 
Since the right branch is closed, it follows that the following rule is admissible in $\JL^\mathcal{T}_\CS$: 

\begin{prooftree}
\AXC{}
\UIC{$c_{i_n}:c_{i_{n-1}}:\ldots:c_{i_1}:A$}
\end{prooftree}
where $c_{i_n}:c_{i_{n-1}}:\ldots:c_{i_1}:A\in\CS$.

\begin{example}
We give a ${\sf J}_\CS^\mathcal{T}$-tableau proof of $x:A \r c \cdot x:(B\r A)$, where $\CS$ contains $c:(A\r (B\r A))$.

\vspace*{0.2cm}
\Tree [.$1.~\neg(x:A\r c\cdot x:(B\r A))$ [.$2.~x:A$ [.$3.~\neg c\cdot x:(B\r A)$ [.$4.~c:(A\r (B\r A))$ {$6.~c\cdot x:(B\r A)$ \\ $\otimes$} ] !\qsetw{5cm} {$5.~\neg c:(A\r (B\r A))$ \\ $\otimes$}  ] ]  ]
\vspace*{0.2cm}

Formulas 2 and 3 are from 1 by rule $(F\r)$, 4 and 5 are obtained by $(PB)$, and 6 from 2 and 4 by rule $(T\cdot)$. Note that in the application of $(PB)$ the $PB$-formula $c:(A\r (B\r A))$ is a ${\sf J}_\CS$-subformula of the root, and in the application of $(T\cdot)$ the formulas $x:A$, $c:(A\r (B\r A))$, and $c\cdot x:(B\r A)$ are ${\sf J}_\CS$-subformula of the root.
\end{example}

Soundness of tableau systems $\JL_\CS^\mathcal{T}$ is a consequence of the following lemma.

\begin{lemma}
Let $\pi$ be any branch of a $\JL^\mathcal{T}_\CS$-tableau and $\M$ be a $\JL_\CS$-model that satisfies all the formulas occur in $\pi$. If a $\JL^\mathcal{T}_\CS$-tableau rule is applied to $\pi$, then it produces at least one extension $\pi'$ such that $\M$ satisfies all the formulas occur in $\pi'$.
\end{lemma}

\begin{theorem}[Soundness]
If $A$ has a $\JL^\mathcal{T}_\CS$-tableau proof, then it is $\JL_\CS$-valid.
\end{theorem}

In order to prove completeness we use the cut rule

\[
\AXC{}\RightLabel{$(cut)$}
\UIC{$A~|~\neg A$}
\DisplayProof
\]

 The cut rule is the same as the principle of bivalence $(PB)$ but without any restrictions on the cut-formula $A$. Completeness is proved by first showing that all theorems of $\JL_\CS$ are provable in the tableau system $\JL_\CS^\mathcal{T} + (cut)$, and then by proving the cut elimination theorem for $\JL_\CS^\mathcal{T} + (cut)$.

\begin{theorem}[Completeness]\label{thm:completeness JL^T+cut}
If $A$ is provable in $\JL_\CS$, then it is provable in the tableau system $\JL_\CS^\mathcal{T} + (cut)$.
\end{theorem}
\begin{proof}
The proof is by induction on the proof of $A$ in $\JL_\CS$. It is a routine matter to check that all axioms of $\JL$ are provable in $\JL_\CS^\mathcal{T}$, even without using $(PB)$ and $(cut)$. If $A$ is obtained from $B$ and $B\r A$ by MP, then by the induction hypothesis there are closed $\JL_\CS^\mathcal{T}$-tableaux $T_1$ and $T_2$ for $B$ and $B\r A$ respectively. Then, using the cut rule twice, the following is a closed tableau for $A$

\vspace*{0.2cm}
\Tree[.$\neg A$  [.$B$ [.$B\r A$ {$\neg B$ \\ $\otimes$} {$A$ \\ $\otimes$} !\qsetw{2cm} ] \qroof{$T_2$}.$\neg (B\r A)$ !\qsetw{3cm} ] \qroof{$T_1$}.$\neg B$ !\qsetw{2.5cm} ]
\vspace*{0.2cm}

Finally, if $A=c:F\in\CS$ is obtained by IAN, then by the closure condition $\neg c:F$ is a closed one-node tableau. \qed
\end{proof}

The proof of the cut elimination is similar to the algorithm given by Fitting in \cite{Fitting1996}, and thus the details will be omitted. The following definitions are inspired from those in \cite{Fitting1996}.

\begin{definition}
The rank of a term $t$ and a formula $A$, denoted by $r(t)$ and $r(A)$ respectively, is defined inductively as follows:
\begin{enumerate}
\item $r(x)=r(c)=0$, for justification variable $x$ and justification constant $c$,\\
 $r(s+t) = r(s\cdot t)= r(s) + r(t) +1$, $r(!t)=r(\bar{?}t)=r(?t) = r(t) +1$.

\item $r(p)=r(\bot)=0$, for $p\in\mathcal{P}$, \\ $r(\neg A)= r(A) +1$, $r(A \r B) = r(A) + r(B) +1$, $r(t:A) = r(t) + r(A) +1$.
\end{enumerate}
\end{definition}

\begin{definition}
Suppose that in a tableau $T$ there is a cut to $A$ and $\neg A$ of the following form:

\vspace*{0.2cm}
\Tree [  \qroof{$T_1$}.$A$ \qroof{$T_2$}.$\neg A$ !{\qbalance} ]
\vspace*{0.2cm}

where $T_1$ and $T_2$ are the subtableaux below $A$ and $\neg A$, respectively. Let $|T|$ denote the number of formulas in the tableau $T$.
\begin{enumerate}
\item We say the cut is at a branch end if $|T_1|=0$ or $|T_2|=0$; that is, if either there are no formulas below $A$, or there are no formulas below $\neg A$, or both.

\item The rank of the cut is the rank of the cut-formula $A$.

\item The weight of the cut is the number of formulas in $T$ strictly below $A$ and $\neg A$; that is, the weight of the cut is $|T_1| + |T_2|$.

\item The cut is called minimal if there are no cuts in the  subtableaux $T_1$ and $T_2$.
\end{enumerate}

\end{definition}

The following fact will be used frequently in the proof of cut elimination (cf. \cite{Fitting1996}). Suppose that $T$ is a closed tableau for a finite set $S$ of formulas and $S \subseteq S'$, where $S'$ is also finite. Then there is a closed tableau for $S'$ with the same number of steps.

\begin{theorem}[Cut Elimination]\label{thm:Cut Elimination}
If a formula is provable in the tableau system $\JL_\CS^\mathcal{T}+ (cut)$, then it is also provable in $\JL_\CS^\mathcal{T}$.
\end{theorem}

\begin{proof}
We will show how to eliminate the minimal cuts from a tableau $T$. Suppose $T$ consists a minimal cut of the following form:

\vspace*{0.2cm}
\Tree [.$\Theta$   \qroof{$T_1$}.$A$ $(cut)$ \qroof{$T_2$}.$\neg A$ !\qsetw{1cm}  !{\brOverride} ]
\vspace*{0.2cm}

 The proof is by induction on the rank of the cut-formula $A$ with subinduction on the weight of the cut. Similar to the cut elimination of the sequent calculus of classical logic (cf. \cite{TS}), we distinguish three cases:

\begin{description}
\item[Case I.] The minimal cut is at a branch end.

\item[Case II.] The minimal cut is not at a branch end, and the uppermost formulas in $T_1$ or $T_2$ are obtained by applying a tableau rule to a formula from $\Theta$.

\item[Case III.] The minimal cut is not at a branch end, and the uppermost formulas in $T_1$ and $T_2$ are obtained by applying  tableau rules to $A$ and $\neg A$, respectively.
\end{description}

In case I, we eliminate the minimal cut. In  cases II and III, we transform the tableau $T$ into another closed tableau in which the minimal cut is replaced by cuts of lower rank, by cuts of the same rank but of lower weight, or both.

{\bf Case I.}
Suppose we have a minimal cut at the end of a branch. We only consider the case in which the branch closes because of $\neg c:F$, where $c:F \in \CS$ (see \cite{Fitting1996} for the other cases). In this case the cut looks like this.

\vspace*{0.2cm}
\Tree [.$\Theta$   \qroof{$T'$}.$c:F$  {$\neg c:F$ \\ $\otimes$} !{\qbalance} ]
\vspace*{0.2cm}

Since $c:F \in \CS$, the cut-formula $c:F$ is a $\JL_\CS$-subformula of the root, and hence the cut is  an instance of $(PB)$.

{\bf Case II.}
Suppose the minimal cut is not at a branch end, and the uppermost formulas in $T_1$ or $T_2$ are obtained by applying a tableau rule to formulas from $\Theta$. In this case we push the cut down in the tableau and obtain a new cut of lower weight. We only consider two cases: (i) the rule $(T\cdot)$ is applied to formulas from $\Theta$, and (ii) the rule $(PB)$ is applied. The other cases are similar.

Suppose the rule $(T\cdot)$ is applied to formulas from $\Theta$. Then the cut is of the form shown in (1), where $s:(A \r B)$, $t:A$, and $s \cdot t:B$ are $\JL_\CS$-subformulas of the root. The displayed cut in (1) is transformed into the one in (2) of lower weight.

(1)
\vspace*{0.2cm}
\Tree [.{$\vdots$ \\ $s:(A \r B)$ \\ $t:A$ \\ $\vdots$}   \qroof{$T_1$}.{$C$ \\ $s \cdot t:B$} \qroof{$T_2$}.$\neg C$ !{\qbalance} ]
\hskip 1.5cm
(2)
\Tree [.{$\vdots$ \\ $s:(A \r B)$ \\ $t:A$ \\ $\vdots$ \\$s \cdot t:B$}   \qroof{$T_1$}.$C$ \qroof{$T_2$}.$\neg C$ !{\qbalance} ]
\vspace*{0.1cm}

Now suppose the rule $(PB)$ is applied. Then the cut is of the form shown in (3), where $A$ is a $\JL_\CS$-subformula of the root. The displayed cut in (3) is transformed into the one in (4) of lower weight.

\vspace*{0.2cm}
(3)
\renewcommand{\qroofpadding}{0.6em}
\Tree [.$\Theta$
 [.$C$ \qroof{$T_1^L$}.$A$ !\qsetw{1cm} $(PB)$ \qroof{$T_1^R$}.$\neg A$ !{\brOverride} ].$C$ !{\brRestore} !\qsetw{4cm}
 \qroof{$T_2$}.$\neg C$ ] 
\hskip 1.5cm
(4)
\Tree [.$\Theta$
 [.$A$  \qroof{$T_1^L$}.$C$ !\qsetw{1.5cm} \qroof{$T_2$}.$\neg C$ ] 
 $(PB)$
 [.$\neg A$  \qroof{$T_1^R$}.$C$  !\qsetw{1.5cm} \qroof{$T_2$}.$\neg C$ ] !{\brOverride} ]
\vspace*{0.1cm}

{\bf Case III.}
Suppose the minimal cut is not at a branch end, and the uppermost formulas in $T_1$ and $T_2$ are obtained by applying  tableau rules to $A$ and $\neg A$, respectively. In this case we transform the cut into cuts of lower rank, or into cuts with the same rank but of lower weight. 

First consider the rule $(T\cdot)$ which is a two-premised rule of the form
\AXC{$\varphi_1$}\noLine
\UIC{$\varphi_2$}
\UIC{$\varphi$}
\DP.
Since $\varphi_1$ is a $\JL_\CS$-subformula of the root, the two cuts to $\varphi_1$ and $\neg\varphi_1$ shown in (5) and (6) are instances of $(PB)$. The same holds if in (5) or (6) a cut  to $\varphi_2$ and $\neg\varphi_2$ is applied.

\vspace*{0.2cm}
(5)
\Tree [.$\Theta$   [.{$\varphi_1$\\ $\varphi_2$} \qroof{$T_1$}.$\varphi$ ] \qroof{$T_2$}.$\neg \varphi_1$ !{\qbalance} ]
\hskip 1.5cm
(6)
\Tree [.{$\Theta$ \\ $\varphi_2$} [.$\varphi_1$ \qroof{$T_1$}.$\varphi$ ] \qroof{$T_2$}.$\neg \varphi_1$ !{\qbalance} ]
\vspace*{0.2cm}

For example, the following cuts are  instances of $(PB)$.

\vspace*{0.2cm}
\Tree [.{$\Theta$ \\ $s':(A \r B)$}
 [.$t+s:A$ \qroof{$T_1$}.$s'\cdot (t+s):B$ ]
 $(PB)$
 [.$\neg t+s:A$ \qroof{$T_2$}.$\neg t:A$ ] !{\qbalance}  !{\brOverride} ].{$\Theta$ \\ $s':(A \r B)$}
 \hskip 1.2cm
 \Tree [.{$\Theta$ \\ $s':A$}
 [.$t+s:(A\r B)$ \qroof{$T_1$}.$(t+s)\cdot s':B$ ]
 $(PB)$
 [.$\neg t+s:(A\r B)$ \qroof{$T_2$}.$\neg t:(A\r B)$  ] !\qsetw{3cm}  !{\brOverride} ]
\vspace*{0.2cm}

\vspace*{0.2cm}
\Tree [.{$\Theta$ \\ $s:(t:A \r B)$}
 [.$!t:t:A$ \qroof{$T_1$}.$s\cdot !t:B$ ]
 $(PB)$
 [.$\neg !t:t:A$ \qroof{$T_2$}.$\neg t:A$ ] !{\qbalance}  !{\brOverride} ].{$\Theta$ \\ $s:(t:A \r B)$}
 \hskip 1.2cm
 \Tree [.{$\Theta$ \\ $s:(\neg t:A \r B)$}
 [.$\bar{?}t:\neg t:A$ \qroof{$T_1$}.$s\cdot \bar{?}t:B$ ]
 $(PB)$
 [.$\neg \bar{?}t:\neg t:A$ \qroof{$T_2$}.$A$ ] !{\qbalance}  !{\brOverride} ].{$\Theta$ \\ $s:(\neg t:A \r B)$}
\vspace*{0.2cm}

\vspace*{0.2cm}
\Tree [.{$\Theta$ \\ $s:(\neg t:A \r B)$}
 [.$\bar{?}t:\neg t:A$ \qroof{$T_1$}.$s\cdot \bar{?}t:B$ ]
 $(PB)$
 [.$\neg \bar{?}t:\neg t:A$ \qroof{$T_2$}.$A$ ] !{\qbalance}  !{\brOverride} ].{$\Theta$ \\ $s:(\neg t:A \r B)$}
\vspace*{0.2cm}

\newpage

Consider the following cut to formulas $\neg t+s :A$ and $\neg\neg t+s :A$ to which the rules $(F+_L)$ and $(F\neg)$ are applied respectively.

\vspace*{0.2cm}
\Tree [.$\Theta$
 [.$\neg t+s:A$ \qroof{$T_1$}.$\neg t:A$ ]
 $(cut)$
 [.$\neg\neg t+s:A$ \qroof{$T_2$}.$t+s:A$ ] !{\qbalance}  !{\brOverride} ].$\Theta$
\vspace*{0.2cm}

This cut is transformed into the following cuts.

\vspace*{0.2cm}
\Tree [.$\Theta$
  [.$t:A$ [.$t+s:A$ {$\neg t+s:A$  \\ $\otimes$} $(cut)_4$ !\qsetw{1.2cm} \qroof{$T_2$}.$\neg\neg t+s:A$ !{\brOverride} ].$t+s:A$ !{\brRestore} !\qsetw{2.5cm}
   $(cut)_2$
     [.$\neg t+s:A$ {$\neg t:A$ \\ $\otimes$} ] !{\brOverride} ].$t:A$ !\qsetw{0.7cm}
    $(cut)_1$ !\qsetw{0.7cm}
   [.$\neg t:A$  \qroof{$T_1$}.$\neg t+s:A$ !{\brRestore}  !\qsetw{2.5cm}
    $(cut)_3$
     [.$\neg\neg t+s:A$ \qroof{$T_2$}.$t+s:A$ ] !{\brOverride} ].$\neg t:A$ !{\brOverride}
     ].$\Theta$
\vspace*{0.2cm}

The rank of $(cut)_1$ and $(cut)_2$ is less than the rank of $(cut)$. Moreover, $(cut)_3$ and $(cut)_4$ have the same rank as $(cut)$ but their weight are smaller than the weight of $(cut)$. The case of $(F+_R)$ is treated in a similar way.

Consider the following cut to formulas $\neg !t:t:A$ and $\neg\neg !t:t:A$ to which the rules $(F!)$ and $(F\neg)$ are applied respectively.

\vspace*{0.2cm}
\Tree [.$\Theta$
 [.$\neg !t:t:A$ \qroof{$T_1$}.$\neg t:A$ ]
 $(cut)$
 [.$\neg\neg !t:t:A$ \qroof{$T_2$}.$!t:t:A$ ] !{\qbalance}  !{\brOverride} ].$\Theta$
\vspace*{0.2cm}

This cut is transformed into the following cuts.


\vspace*{0.2cm}
\Tree [.$\Theta$
  [.$t:A$ 
  [.$!t:t:A$ {$\neg !t:t:A$ \\  $\otimes$}  !\qsetw{2cm}
   $(cut)_4$
    \qroof{$T_2$}.$\neg\neg !t:t:A$ !{\brOverride} ].$!t:t:A$ !{\brRestore} 
 !\qsetw{3cm}   $(cut)_2$ !\qsetw{0.5cm}
     [.$\neg !t:t:A$ {$\neg t:A$ \\  $\otimes$} ] !{\brOverride} ].$t:A$ !\qsetw{3cm}
    $(cut)_1$ !\qsetw{1cm}
   [.$\neg t:A$  \qroof{$T_1$}.$\neg !t:t:A$ !{\brRestore}  !\qsetw{2.5cm}
    $(cut)_3$
     [.$\neg\neg !t:t:A$ \qroof{$T_2$}.$!t:t:A$ ] !{\brOverride} ].$\neg t:A$ !{\brOverride}
     ].$\Theta$
\vspace*{0.2cm}

The rank of $(cut)_1$ and $(cut)_2$ is less than the rank of $(cut)$. Moreover, $(cut)_3$ and $(cut)_4$ have the same rank as $(cut)$ but their weight are smaller than the weight of $(cut)$. The  cut to formulas $\neg ?t:\neg t:A$ and $\neg\neg ?t:\neg t:A$ to which the rules $(F?)$ and $(F\neg)$ are applied respectively is treated similarly.

Consider the following cut to formulas $\neg \bar{?}t:\neg t:A$ and $\neg\neg \bar{?}t:\neg t:A$ to which the rules  $(F\bar{?})$ and $(F\neg)$ are applied respectively.

\vspace*{0.2cm}
\Tree [.$\Theta$
 [.$\neg\bar{?}t:\neg t:A$ \qroof{$T_1$}.$A$ ]
 $(cut)$
 [.$\neg\neg \bar{?}t:\neg t:A$ \qroof{$T_2$}.$\bar{?}t:\neg t:A$ ] !{\qbalance}  !{\brOverride} ].$\Theta$
\vspace*{0.2cm}

This cut is transformed into the following cuts.


\vspace*{0.2cm}
\Tree [.$\Theta$
  [.$A$  
  \qroof{$T_1$}.$\neg \bar{?}t:\neg t:A$ !\qsetw{2.5cm}
   $(cut)_2$
    [.$\neg\neg \bar{?}t:\neg t:A$ \qroof{$T_2$}.$\bar{?}t:\neg t:A$ ] !{\brOverride} ].$A$  
   $(cut)_1$ !\qsetw{0.01cm}
   [.$\neg A$  
   [.$\bar{?}t:\neg t:A$  {$\neg \bar{?}t:\neg t:A$ \\ $\otimes$  } !\qsetw{2.5cm}
    $(cut)_4$
     \qroof{$T_2$}.$\neg\neg \bar{?}t:\neg t:A$ !{\brOverride} ].$\bar{?}t:\neg t:A$ !{\brRestore}
   !\qsetw{1.5cm}   $(cut)_3$ !\qsetw{1.5cm}
     [.$\neg \bar{?}t:\neg t:A$ {$A$ \\ $\otimes$ }  ]  !{\brOverride} ].$\neg A$  !{\brOverride} ].$\Theta$ 
\vspace*{0.2cm}

The rank of $(cut)_1$ and $(cut)_3$ is less than the rank of $(cut)$. Moreover, $(cut)_2$ and $(cut)_4$ have the same rank as $(cut)$ but their weight are smaller than the weight of $(cut)$.

Now suppose that jT is an axiom of $\JL$. Consider the following cut to formulas $t+s :A$ and $\neg t+s :A$ to which the rules $(T:)$ and $(F+_L)$ are applied respectively.

\vspace*{0.2cm}
\Tree [.$\Theta$
 [.$t+s:A$ \qroof{$T_1$}.$A$ ]
 $(cut)$
 [.$\neg t+s:A$ \qroof{$T_2$}.$\neg t:A$ ] !{\qbalance}  !{\brOverride} ].$\Theta$
\vspace*{0.2cm}

This cut is transformed into the following cuts.


\vspace*{0.2cm}
\Tree [.$\Theta$
  [.$t:A$  [.$A$ \qroof{$T_1$}.$t+s:A$ $(cut)_4$  [.$\neg t+s:A$ {$\neg t:A$ \\  $\otimes$}   ]  !{\brOverride} ].$A$  !{\brRestore} !\qsetw{2.5cm}
   $(cut)_2$
    [.$\neg A$ {$A$ \\ $\otimes$} ] !{\brOverride} ].$t:A$ !\qsetw{4.5cm}
    $(cut)_1$
   [.$\neg t:A$  [.$t+s:A$ \qroof{$T_1$}.$A$ !{\brRestore} ].$t+s:A$  !\qsetw{2.5cm}
    $(cut)_3$
     \qroof{$T_2$}.$\neg t+s:A$ !{\brOverride} ].$\neg t:A$ !{\brOverride} ].$\Theta$
\vspace*{0.2cm}

The rank of $(cut)_1$ and $(cut)_2$ is less than the rank of $(cut)$. Moreover, $(cut)_3$ and $(cut)_4$ have the same rank as $(cut)$ but their weight are smaller than the weight of $(cut)$. The case of $(F+_R)$ is similar.

Consider the following cut to formulas $!t:t:A$ and $\neg !t:t:A$ to which the rules $(T:)$ and $(F!)$ are applied respectively.

\vspace*{0.2cm}
\Tree [.$\Theta$
 [.$!t:t:A$ \qroof{$T_1$}.$t:A$ ]
 $(cut)$
 [.$\neg !t:t:A$ \qroof{$T_2$}.$\neg t:A$ ] !{\qbalance}  !{\brOverride} ].$\Theta$
\vspace*{0.2cm}

This cut is transformed into the following cuts.

\vspace*{0.2cm}
\Tree [.$\Theta$
  [.$t:A$  \qroof{$T_1$}.$!t:t:A$  !\qsetw{2.5cm}
   $(cut)_2$
    [.$\neg !t:t:A$ {$\neg t:A$ \\ $\otimes$} ] !{\brOverride} ].$t:A$ !\qsetw{4.5cm}
    $(cut)_1$
   [.$\neg t:A$  [.$!t:t:A$ {$ t:A$ \\ $\otimes$} !{\brRestore} ].$!t:t:A$  !\qsetw{2.5cm}
    $(cut)_3$
     \qroof{$T_2$}.$\neg !t:t:A$ !{\brOverride} ].$\neg t:A$ !{\brOverride} ].$\Theta$
\vspace*{0.2cm}

The rank of $(cut)_1$ is less than the rank of $(cut)$. Moreover, $(cut)_2$ and $(cut)_3$ have the same rank as $(cut)$ but their weight are smaller than the weight of $(cut)$.
The  cut to formulas $?t:\neg t:A$ and $\neg ?t:\neg t:A$ to which the rules $(T:)$ and $(F?)$ are applied respectively is treated similarly.

Consider the following cut to formulas $\bar{?}t:\neg t:A$ and $\neg \bar{?}t:\neg t:A$ to which the rules $(T:)$ and $(F\bar{?})$ are applied respectively.

\vspace*{0.2cm}
\Tree [.$\Theta$
 [.$\bar{?}t:\neg t:A$ \qroof{$T_1$}.$\neg t:A$ ]
 $(cut)$
 [.$\neg \bar{?}t:\neg t:A$ \qroof{$T_2$}.$A$ ] !{\qbalance}  !{\brOverride} ].$\Theta$
\vspace*{0.2cm}

This cut is transformed into the following cuts.

\vspace*{0.2cm}
\Tree [.$\Theta$
  [.$t:A$  [.$A$  [.$\bar{?}t:\neg t:A$ {$\neg t:A$ \\ $\otimes$} ] !\qsetw{1.8cm} 
$(cut)_4$ !\qsetw{1.2cm} 
\qroof{$T_2$}.$\neg \bar{?}t:\neg t:A$ !{\brOverride} ].$A$ !{\brRestore} !\qsetw{3.1cm} 
   $(cut)_2$ !\qsetw{0.1cm} 
   [.$\neg A$ {$A$ \\ $\otimes$} ] !{\brOverride} ].$t:A$ !\qsetw{0.3cm}
    $(cut)_1$  !\qsetw{0.3cm}
   [.$\neg t:A$   \qroof{$T_1$}.$\bar{?}t:\neg t:A$  !{\brRestore}   !\qsetw{2.5cm}
    $(cut)_3$
     [.$\neg \bar{?}t:\neg t:A$ \qroof{$T_2$}.$A$ ] !{\brOverride} ].$\neg t:A$ !{\brOverride} ].$\Theta$
\vspace*{0.2cm}

The rank of $(cut)_1$ and $(cut)_2$ is less than the rank of $(cut)$. Moreover, $(cut)_3$ and $(cut)_4$ have the same rank as $(cut)$ but their weight are smaller than the weight of $(cut)$.

Consider the cut  to formulas $t+s:\bot$ and $\neg t+s:\bot$, shown in (7), to which the rules $(T:_\bot)$ and $(F+_L)$ are applied respectively.

\vspace*{0.2cm}
(7)
\Tree [.$\Theta$   [.$t+s:\bot$  {$\bot$ \\ $\otimes$} ] [.$\neg t+s:\bot$  \qroof{$T'$}.$\neg t:\bot$ ] !{\qbalance} ]
\hskip 1.5cm
(8)
\Tree [.$\Theta$    [.$t:\bot$ {$\bot$ \\ $\otimes$} ]
[.$\neg t:\bot$ [.$t+s:\bot$ {$\bot$ \\ $\otimes$} ]  \qroof{$T'$}.$\neg t+s:\bot$   !\qsetw{2.5cm} ]  !\qsetw{4.5cm} ]
\vspace*{0.2cm}

 The cut in (7) is transformed into the cuts shown in (8), in which the cut to $t:\bot$ and $\neg t:\bot$ has a lower rank, and the weight of the cut  to $t+s:\bot$ and $\neg t+s:\bot$ is smaller than the weight of the original cut. The case of $(F+_R)$ is treated in a similar way.

Actually there are two remaining cuts to verify in this case: the cut to formulas $A \r B$ and $\neg (A \r B)$ to which the rules $(T\r)$ and $(F\r)$ are applied respectively; and the cut to formulas $\neg (A \r B)$ and $\neg\neg (A \r B)$ to which the rules $(F\r)$ and $(F\neg)$ are applied respectively. We refer the reader to \cite{Fitting1996}  for a more detailed exposition of these two cuts. \qed
\end{proof}

\begin{theorem}[Completeness]\label{thm:completeness KE-tableaux}
If $A$ is $\JL_\CS$-valid, then it has a $\JL^\mathcal{T}_\CS$-tableau proof.
\end{theorem}
\begin{proof}
If $A$ is $\JL_\CS$-valid, then by Theorem \ref{thm:Sound Compl JL} it is provable in $\JL_\CS$. Hence, by Theorem \ref{thm:completeness JL^T+cut}, it is provable in $\JL^\mathcal{T}_\CS + (cut)$. Then, by the cut elimination theorem, it is provable in $\JL^\mathcal{T}_\CS$. \qed
\end{proof}

Inspection of all $\JL^\mathcal{T}_\CS$-tableau rules in Tables \ref{table:J^T} and \ref{table:tableau rules JL}   shows that in a $\JL^\mathcal{T}_\CS$-tableau every expanded formula of a rule is a weak $\JL_\CS$-subformula of the root of the tableau.

\begin{theorem}[Subformula property]\label{thm:subformula property tableaux}
Every formula in a $\JL^\mathcal{T}_\CS$-tableau proof is a weak $\JL_\CS$-subformula of the root of the tableau.
\end{theorem}

Note that the subformula property does not ensure decidability, because the number of $\JL_\CS$-subformulas of a formula is not necessarily finite. In fact, it is wrongly claimed in \cite[page 172]{Finger2010} that for a finite $\CS$, the set of all $\JL_\CS$-subformulas of a formula is always finite. For a counterexample, consider a formula $t:A$ and an empty $\CS$. The set of all $\JL_\emptyset$-subformulas of $t:A$ includes $t:A, t:t:A, t:t:t:A, \ldots$, which is obviously infinite.

\section{Conclusion}
We introduced two kinds of tableau proof systems for each justification logic $\JL$, i.e. $\JL$-tableaux of Section \ref{sec:Tableaux 1} and $\JL^\mathcal{T}$-tableaux of Section \ref{sec:JL^T tableaux}. We proved soundness and completeness theorems for both kinds of tableaux. While some $\JL$-tableau rules are not analytic, we showed a kind of subformula property for $\JL^\mathcal{T}$-tableaux.   \\

\noindent
{\bf Acknowledgments}\\

This research was in part supported by a grant from IPM. (No. 95030416)


\end{document}